\documentclass{amsart}
 \usepackage[margin=2.5cm]{geometry}
\usepackage{enumerate,color}

\newtheorem{theorem}{Theorem}[section]
\newtheorem{lemma}[theorem]{Lemma}

\renewcommand{\Re}{\operatorname{Re}}

\newcommand{\U}{ \mathcal{U}}

\newtheorem{thmA}{Theorem}

\renewcommand{\Re}{\operatorname{Re}}
\renewcommand{\Im}{\operatorname{Im}}

\begin{document}
	\title{Universality of the zeta function in short intervals}
	
	\author{Yoonbok Lee}
	\address{Department of Mathematics, Incheon National University, Incheon 22012, Korea}
	\email{leeyb131@gmail.com, leeyb@inu.ac.kr}
	
	\author{{\L}ukasz Pa\'nkowski}
	\address{Faculty of Mathematics and Computer Science, Adam Mickiewicz University, Uniwersytetu Pozna\'nskiego 4, 61-614 Pozna\'{n}, Poland}
	\email{lpan@amu.edu.pl}

	\thanks{The first author was supported by the National Research Foundation of Korea (NRF) grant funded by the Korea government (MSIT) (No. RS-2024-00415601). The second author was partially supported by the grant no. 2021/41/B/ST1/00241 from the National Science Centre, Poland.}
	
	
	
	\begin{abstract}
We improve the universality theorem of the Riemann zeta-function in short intervals by establishing universality for significantly shorter intervals $[T,T+H]$. Assuming the Riemann Hypothesis, we prove that universality in such short intervals holds for $H=(\log T)^B$ with an explicitly given $B>0$. Unconditionally, we show that for the same $H$ the set of real numbers $\tau\in[T,T+H]$ such that $\zeta(s+i\tau)$ approximates an arbitrary given analytic function has a positive upper density.
	\end{abstract}
	
	\maketitle

\section{Introduction}

In 1975, Voronin \cite{V1} proved the remarkable property of the Riemann zeta-function, known as the universality theorem. Let 
\begin{align*} 
	\U &:= \{s\in\mathbb{C}:\frac12 <\Re(s)<1\},\\
D_r(s_0) &:= D_r ( \sigma_0 + i t_0 )  := \{s\in\mathbb{C}:|s-s_0|\leq r\}
\end{align*}
throughout the paper. Let $ D= D_r ( 3/4) \subset \U$. Then he proved that for any continuous non-vanishing function $f(s)$ on $D $, analytic in the interior of $D $, we have
\begin{equation}\label{eq:Voronin_classical}
 \liminf_{T\to\infty}\frac1T \mu \left\{\tau\in [0,T]:\max_{s\in D}|\zeta(s+i\tau)-f(s)|<\varepsilon\right\}>0
\end{equation}
for every $\varepsilon>0$, where $\mu\{\cdot\}$ denotes the Lebesgue real measure. 

Voronin's theorem has been extended to many other zeta and $L$-functions, and research on universality has been conducted in many different directions. For more details, we refer to a comprehensive survey paper by Matsumoto \cite{Mat}.  One of the most significant directions is the problem of effectivization of Voronin's theorem, which was already investigated by Voronin in \cite{V5}. Despite the fact that, up to now, there are various results on this topic (see for example \cite{Gar, GLMSS, Good, L_effec1, L_effec2, S_effec1, S_effec2}), the problem is still far from being well understood. For instance, Laurin\v{c}ikas in \cite{L_short} proposed finding the shortest possible interval of the form $[T,T+H]$ with $H=o(T)$ as $T\to\infty$, for which \eqref{eq:Voronin_classical} holds for $[T,T+H]$ instead of $[0,T]$. He managed to identify that the main constraint on $H$ came from proving a bounded mean-square for $\zeta(s)$, that is, 
\begin{equation}\label{eq:meanSquare}
\frac{1}{H}\int_T^{T+H}|\zeta(\sigma+it)|^2dt\ll_\sigma 1,\qquad \sigma>\frac{1}{2}.
\end{equation}
Since it is well-known (see \cite[Theorem 7.1]{Ivic}) that one can deduce \eqref{eq:meanSquare} from the theory of exponent pairs (see \cite[Section 2.3]{Ivic}), Laurincikas succeeded in proving the following result.
\begin{thmA}[{\cite[Theorem 4]{L_short}}]\label{thm:LaurincikasShort}
Suppose that $T^{1/3}(\log T)^{26/15}\leq H\leq T$. Let $K\subset\U $ be a compact set with connected complement and $f(s)$ be a non-vanishing continuous function on $K$, analytic in the interior of $K$. Then, for every $\varepsilon>0$, we have
\[
\liminf_{T\to\infty} \frac1H \mu \left\{\tau\in [T,T+H]:\max_{s\in K}|\zeta(s+i\tau)-f(s)|<\varepsilon\right\}>0.
\]
\end{thmA}
The more detailed discussion on  optimal choice of an exponent pair was made in \cite{AGKNPSSSW}, where Andersson et al. showed that in fact one can get a slight improvement of Theorem \ref{thm:LaurincikasShort}, since \eqref{eq:meanSquare}, and in consequence Theorem \ref{thm:LaurincikasShort} too, holds for every $H$ satisfying $T^{23/70}(\log T)^{61/35}\leq H\leq T$. However, one can get the best restriction on $H$ by using a recent result of Bourgain and Watt \cite{BW}, since, as it was showed in \cite[Theorem~1]{AGKNPSSSW}, Theorem~\ref{thm:LaurincikasShort} holds when $T^{1273/4053}\leq H\leq T$. Moreover, they noticed that the smaller $H$ is admissible, if a compact set $K$ is not too close to the critical line. For example, Theorem \ref{thm:LaurincikasShort} holds for every $H$ satisfying $T^{9/35}(\log T)^{61/35}\leq H\leq T$, provided $K\subset\{s\in\mathbb{C}:31/52<\Re(s)<1\}$.

Section 4 in \cite{AGKNPSSSW} is devoted to an interesting discussion addressing the question of how small conjecturally the value $H$ can be. Since \eqref{eq:meanSquare} implies Theorem \ref{thm:LaurincikasShort}, it is almost an immediately consequence of the classical Lindel\"of hypothesis that Theorem \ref{thm:LaurincikasShort} holds for every $H$ satisfying $T^\varepsilon\leq H\leq T$ for any $\varepsilon>0$. Similarly, assuming the Riemann Hypothesis for $\zeta(s)$, which is a stronger conjecture, gives us the truth of Theorem \ref{thm:LaurincikasShort} for every $H$ satisfying $\exp((\log T)^{1-\varepsilon})\leq H\leq T$ for every $\varepsilon>0$ (see \cite[Theorem 4]{AGKNPSSSW}). The main purpose of the paper is to show that   the Riemann Hypthesis allows us to take much smaller $H$ in Theorem \ref{thm:LaurincikasShort}.
Here is what we have on a closed disc assuming RH.
\begin{theorem}\label{main theorem}
	Suppose that the Riemann Hypothesis holds. Let $D=D_r(s_0)  \subset\U$ and $H=(\log T)^B$ with $B>\frac{2}{2(\sigma_0  -r) -1}$. Assume that $f(s)$ is a continuous function on $D$, analytic in the interior of $D$. Then, for every $\varepsilon>0$, we have
	\[
	\liminf_{T\to\infty}\frac{1}{H}\mu\left\{\tau\in[T,T+H]: \max_{s\in D}|\zeta(s+i\tau) - f(s)|<\varepsilon\right\}>0.
	\]
\end{theorem}
In contrast to the previously mentioned approaches for establishing universality in short intervals, which focus on improving the bounded mean-square for the Riemann zeta-function, namely \eqref{eq:meanSquare}, our proof involves approximating the logarithm of the Riemann zeta function using the shortest possible Dirichlet polynomial as in Lemma \ref{lemma GS}. Consequently, the constraint on 
$H$ will be directly determined by the number of terms in a Dirichlet polynomial of minimal length that can effectively approximate $\log\zeta(s)$ with $1/2<\Re(s)<1$.

Lemma \ref{lemma GS} tells us that $ \log \zeta(s)$ has a short Dirichlet polynomial approximation when $s $ is away from any zeros of $\zeta(s)$. It is well-known that  there are not too many zeros off the critical line in the sense that    
\begin{equation}\label{eqn:zerodensity}
  N( \sigma ,T) \ll T^{ \frac32 - \sigma  } ( \log T)^5 
  \end{equation}
by \cite[Theorem 9.19 A]{T}, where $N( \sigma ,T ) $ is the number of zeros $ \rho = \beta + i \gamma $ of the Riemann zeta function with $ \beta > \sigma $ and  $ 0< \gamma <T $. By \eqref{eqn:zerodensity} there exist infinitely many zero free blocks unconditionally, so that we prove a universality in the blocks as follows. 
\begin{theorem}\label{main theorem uncond}
	 Let $D=D_r(s_0)  \subset \U$ and $H=(\log T)^B$ with $B>\frac{4}{2(\sigma_0  -r) -1}$. Assume that $f(s)$ is a continuous function on $D$, analytic in the interior of $D$. Then, for every $\varepsilon>0$, we have
	\[
	\limsup_{T\to\infty}\frac{1}{H}\mu\left\{\tau\in[T,T+H]: \max_{s\in D}|\zeta(s+i\tau) - f(s)|<\varepsilon\right\}>0.
	\]
\end{theorem}

Our approach cannot prove Theorem \ref{main theorem} unconditionally, because it is not possible to find a Dirichlet polynomial approximation to $\log \zeta(s)$ around nontrivial zeros of $\zeta(s)$. One may improve our theorems replacing $D$ by compact sets with connected complements without much difficulties by following references. However, improving $H$ shorter or replacing $\limsup$ by $ \liminf$  in Theorem \ref{main theorem uncond} would be challenging.

 In Section \ref{section lemmas} we introduce required lemmas. Then we prove Theorem \ref{main theorem} in Section \ref{section proof 1} and Theorem \ref{main theorem uncond} in Section \ref{section proof 2}.

\section{Auxiliary lemmas}\label{section lemmas}

The first lemma provides a short Dirichlet polynomial approximation to $\log \zeta(s)$ provided that there are no zeros of $\zeta(s) $ around $ s $. 

\begin{lemma}[{\cite[Lemma 1]{GS}}] \label{lemma GS}
Let $ y \geq 2 $ and $ |t| \geq y+3$ be real numbers. Let $ \frac12 \leq \sigma_0 < 1 $ and suppose that the rectangle $ \{ z : \sigma_0 < \Re(z) \leq 1 , | \Im (z) - t | \leq y+2 \} $ is free of zeros of $ \zeta(z)$. Then for any $ \sigma_0 < \sigma \leq 2 $ and $ | \xi - t | \leq y $ we have
$$ | \log \zeta ( \sigma + i \xi ) | \ll \log |t| \log ( e/ ( \sigma - \sigma_0 ) ) . $$
Further for $ \sigma_0 < \sigma \leq 1 $ we have
$$ \log \zeta( \sigma + i t ) = \sum_{n \leq y} \frac{ \Lambda(n)}{ n^{ \sigma + it } \log n } + O \bigg(   \frac{ \log |t|}{ ( \sigma_1 - \sigma_0 )^2 } y^{ \sigma_1 - \sigma } \bigg) , $$
where we put $ \sigma_1 = \min \bigg( \sigma_0 + \frac{1}{ \log y} , \frac{ \sigma+ \sigma_0 }{2} \bigg) $.\pagebreak
\end{lemma}

By modifying Lemma \ref{lemma GS}, we find suitable representations of $ \log \zeta(s)$ for $ s \in \U$ for our proof of universality. 

\begin{lemma}\label{lem:shortDirichletpoly}
Assume RH.   Let $ \frac12 < \sigma_1 <1 $, $A>0$ and $ \widetilde{T} >10$ be fixed real numbers.  Let  $ Y = ( \log T)^A $, then
we have
$$  \log \zeta ( \sigma + it ) =  - \sum_{ p \leq Y  } \log \left( 1 - \frac{1}{ p^{ \sigma + it }} \right) + O\left(   ( \log T)^{1+ A (\frac12 - \sigma_1) }( \log \log T)^2   \right)   $$
for $ \sigma_1  \leq  \sigma \leq 1 $ and $ T-1 \leq t - \widetilde{T} \leq 2T +1$ as $T \to \infty$.
\end{lemma}

\begin{lemma}\label{lem:shortDirichletpoly uncond}
 Let $ \frac12 < \sigma_1 <1 $, $A,B>0$ and $ \widetilde{T} >10$ be fixed real numbers. There exists an increasing sequence $ T_j  \to \infty$ such that 
$$  \log \zeta ( \sigma + it ) =  - \sum_{ p \leq Y_j  } \log \left( 1 - \frac{1}{ p^{ \sigma + it }} \right) + O\left(   ( \log T_j )^{1+ A (\frac12 - \sigma_1) }( \log \log T_j )^2   \right)   $$
for $ \sigma_1  \leq  \sigma \leq 1 $, $Y_j = ( \log T_j )^A$, $H_j = ( \log T_j )^B$  and $ T_j -1 \leq t -\widetilde{T} \leq T_j + H_j+1 $ as $T_j  \to \infty$.
\end{lemma}

The difficulty in the unconditional case comes from the branches of $ \log \zeta(s)$. It is resolved by means of the zero density estimate \eqref{eqn:zerodensity}. 
 Now we prove Lemmas \ref{lem:shortDirichletpoly} and \ref{lem:shortDirichletpoly uncond} both.

\begin{proof}[Proof of Lemma \ref{lem:shortDirichletpoly}]
By Lemma \ref{lemma GS} with $ \sigma_0 = \frac12$ and $y = Y= ( \log T)^A$,  we have
$$ \log \zeta(s ) = \sum_{ p^k \leq Y} \frac{1}{ k p^{ks}} + O ( ( \log T)^{1+ A( \frac12 - \sigma)} ( \log \log T)^2 ) $$
for $ \sigma > \frac12 $ and $ T \ll t \ll T$. 
Since
\begin{equation}\label{eqn:prime sum difference}
\begin{split}
	- \sum_{p\leq Y}\log\left(1-\frac{1}{p^{s}}\right)  - \sum_{ p^k \leq Y} \frac{1}{ k p^{ks}} &= \sum_{p\leq Y}\sum_{k>\tfrac{\log Y}{\log p}}\frac{1}{kp^{ks}}\\
	&= \sum_{p\leq \sqrt{Y}}\sum_{k>\tfrac{\log Y}{\log p}}\frac{1}{kp^{ks}}+ \sum_{\sqrt{Y}<p\leq Y}\sum_{k\geq 2}\frac{1}{kp^{ks}} \ll  \frac{Y^{\tfrac{1}{2}-\sigma}}{\log Y}  ,
\end{split}
\end{equation}
the lemma follows.
\end{proof}

\begin{proof}[Proof of Lemma \ref{lem:shortDirichletpoly uncond}]
By \eqref{eqn:zerodensity} we have
$$  N \left(   \frac14 + \frac{ \sigma_1}{2} , 2T  \right)  \ll T^{ \frac54 - \frac{ \sigma_1}{2}  } ( \log T)^5. $$
Thus, for every sufficiently large $T$, there exists $ T_j \in [T, 2T]$ such that there are no zeros of $\zeta(s)$ on $ \Re(s) > \frac14 + \frac{ \sigma_1}{2}$ and $ | \Im (s) - T_j | \leq  2 + Y_j  + 2H_j   $ with $  Y_j = ( \log T_j )^A $. By Lemma \ref{lemma GS} we have
$$ \log \zeta(s ) = \sum_{ p^k \leq Y_j } \frac{1}{ k p^{ks}} + O ( ( \log T_j )^{1+ A( \frac12 - \sigma_1 )} ( \log \log T_j )^2 ) $$
for $ \sigma_1 \leq \sigma \leq 1 $ and $ | t - T_j | \leq 2 H_j $. 
Applying \eqref{eqn:prime sum difference} to the above equation proves the lemma.
\end{proof}

We list three lemmas from \cite{K-V} with minor modifications and without proofs. The first is to find a suitable approximation for $\log  f(s) $ in Theorems \ref{main theorem} and \ref{main theorem uncond}. 

\begin{lemma}[{Lemma 1 in \cite[Chapter VII, \S1]{K-V}}]\label{lem:denseness}
Let $D=D_r(s_0)  \subset\U$. Suppose that $g(s)$ be continuous on $D$ and analytic on the interior of $D$. Then, for every $\varepsilon>0$ and $y>0$, there is a finite set of primes $M\supset \{p:p\leq y\}$ such that
\[
\max_{s\in D}\left|g(s)-\sum_{p\in M}\log\left(1-\frac{e(-\theta_p)}{p^s}\right)^{-1}\right|<\varepsilon
\]
for some sequence $(\theta_p)_{p}$ of real numbers indexed by primes, where as usual $e(x) = \exp(2\pi i x)$.
\end{lemma}

We will apply Lemma \ref{lem:denseness} to $g(s) = \log f(s)$. Then we need to find a connection between $p^{i\tau} = e\left( \frac{ \tau   \log p  }{ 2 \pi } \right) $ in Lemmas \ref{lem:shortDirichletpoly} and \ref{lem:shortDirichletpoly uncond} and $ e( - \theta_p)$ in Lemma \ref{lem:denseness}. To be more specific, define 
\begin{equation}\label{def C delta}\begin{split}
	C(\delta,M,T) &= \left\{\tau\in [T,T+H]: \left\Vert\theta_p - \tfrac{\tau\log p}{2\pi}\right\Vert<\tfrac{\delta}{2}\text{ for }p\in M \right\}, \\
	C(\delta,M) &=\left\{(\vartheta_p)_p\in\Omega: \left\Vert\theta_p -\vartheta_p\right\Vert<\tfrac{\delta}{2}\text{ for }p\in M \right\},
\end{split}\end{equation} 
where $M$ is a finite set of primes and  $\Omega$ is the set of all sequences of real numbers in $[0, 1]$ indexed by the primes. Assume that $ 0 < H \leq T$ and $ H \to \infty $ as $ T \to \infty$. Then by Lemma \ref{lem:KroneckerWeyl} below we have
\begin{equation}\label{eqn:C delta asymp}
\mu(C(\delta,M, T))\sim H\mu(C(\delta,M)) = \delta^{|M|} H,
\end{equation}
where $|M|$ denotes the number of elements of the set $M$.
  
\begin{lemma}[Kronecker-Weyl]\label{lem:KroneckerWeyl}
Let $\alpha_1,\ldots,\alpha_n\in\mathbb{R}$ be linearly independent over $\mathbb{Q}$, and $\gamma\subset[0,1]^n$ be the set with Jordan volume $\Gamma$. Assume that $ 0 < H = H(T) \leq T $ and $ H \to \infty $ as $ T \to \infty$ and let  $I_\gamma(T)  = \mu\{\tau\in [T,T+H]: (\tau\alpha_1,\ldots,\tau\alpha_n)\in \gamma\mod{1}\}$. Then
\[
\lim_{T\to\infty}\frac{I_\gamma(T)}{H} = \Gamma.
\] 
\end{lemma}
 One can find a proof of Lemma \ref{lem:KroneckerWeyl} by modifying the proof of Theorem 1 in \cite[Appendix, \S 8]{K-V}. We also need a minor modification of Theorem 3 in \cite[Appendix, \S 8]{K-V} and write it as a next lemma.

\begin{lemma}\label{lem:unifDistCurv}
Suppose that the curve $\gamma(\tau)$ is uniformly distributed mod $1$ in short intervals in $\mathbb{R}^n$. Let $D$ be a closed and Jordan measurable subregion of the unit cube in $\mathbb{R}^n$, and let $\Omega$ be a family of complex-valued continuous functions defined on $D$. If $\Omega$ is uniformly bounded and equicontinuous, then the following relation holds uniformly with respect to $F \in\Omega$:
\[
\lim_{T\to\infty}\frac{1}{H}\int_{\mathcal{S} } F (\{\gamma(\tau)\})d\tau = \int_D F(x_1 , \ldots, x_n )  dx_1\cdots dx_n,
\]
where $ 0 < H \leq T$, $H\to\infty$ as $T\to\infty$,  $\mathcal{S} = \{  \tau\in[T,T+H]  : \gamma(\tau)\in D \mod{1}  \} $, and $\{\gamma(\tau)\} = (\{\gamma_1(\tau)\} ,\ldots,\{\gamma_n(\tau)\})$.
\end{lemma}
 Note that the curve $\gamma(\tau)$ is uniformly distributed mod $1$ in short intervals if for every $\Pi = [a_1,b_1]\times\cdots\times[a_n,b_n]$, $0\leq a_j<b_j\leq 1$ for $j=1,2,\ldots,n$, it holds
\[
\lim_{T\to\infty}\frac{1}{H}\mu\left\{\tau\in[T,T+H]: \gamma(\tau)\in \Pi\mod{1}\right\}=\prod_{j=1}^n(b_j-a_j)
\]
for $0 < H \leq T $ and $ H \to \infty $ as $T\to\infty$.

\section{Proof of Theorem \ref{main theorem}}\label{section proof 1}

Let $ H = ( \log T)^B$ and $Y = ( \log T)^A$ with $ B>A> \frac{2}{ 2( \sigma_0 -r)-1}$. By Lemma \ref{lem:shortDirichletpoly}   there exists $ c_1>0$ such that 
\begin{equation}\label{eqn:proof inequality 1}
\max_{s\in D} \left|\log\zeta(s+i\tau)-\sum_{p\leq Y}\log\left(1-\frac{1}{p^{s+i\tau}}\right)^{-1}\right| \leq (\log T)^{-c_1}
\end{equation}
for all $\tau\in[T,T+H]$. Let $\varepsilon>0$ and $ y >0$.   By Lemma \ref{lem:denseness} with $ g(s) = \log f(s)$, there is a finite set of primes $M$ containing $ \{p:p\leq y\}$ such that
\begin{equation}\label{eqn:proof inequality 2}
\max_{s\in D}\left|\log f(s) - \sum_{p\in M}\log\left(1-\frac{e(-\theta_p)}{p^s}\right)^{-1}\right|<\varepsilon
\end{equation}
for some $(\theta_p)_p \in \Omega$.
By \eqref{def C delta} and uniform continuity, there exists $ \delta >0 $ such that 
\begin{equation}\label{eqn:proof inequality 3}
\max_{s\in D}\left|\sum_{p\in M}\log\left(1-\frac{e(-\theta_p)}{p^s}\right) - \sum_{p\in M}\log\left(1-\frac{1}{p^{s+i\tau}}\right)\right|<\varepsilon
\end{equation}
for $\tau \in C(\delta,M,T)$. Note that by \eqref{eqn:C delta asymp} such $ \tau \in C(\delta, M,T)$ are a positive proportion of $ [T, T+H]$.

Let $D'=D_R(s_0) \subset \U $ with $R=\tfrac{1}{2}(\sigma_0+r-\tfrac{1}{2})>r$ and $\sigma'_1=\min_{s\in D'}\Re(s) = \sigma_0 - R $.  Let us choose 
\begin{equation}\label{choice of y star}
y^*>\max\{y\delta^{|M|(1-2\sigma'_1)^{-1}},\max\{p:p\in M\}\}
\end{equation}
 and $P$ be the greatest prime number $\leq y^*$.  
By Lemma \ref{lem:unifDistCurv} we obtain
\begin{align*}
	 \lim_{T\to\infty}\frac{1}{H}\int_{C(\delta,M,T)}\bigg|\sum_{p\in M}\log\left(1-\frac{1}{p^{s+i\tau}}\right) & -\sum_{p\leq y^*}\log\left(1-\frac{1}{p^{s+i\tau}}\right)\bigg|^2 d\tau\\
	& =\int_{C(\delta,M)}\bigg|\sum_{\substack{p\leq y^*\\p\not\in M}}\log\left(1-\frac{e(-\vartheta_p)}{p^{s}}\right)\bigg|^2 d\vartheta_2\cdots d \vartheta_P\\
	& =\delta^{|M|}\int_0^1\cdots\int_0^1 \bigg|\sum_{\substack{p\leq y^*\\p\not\in M}}\log\left(1-\frac{e(-\vartheta_p)}{p^{s}}\right)\bigg|^2 \bigg( \prod_{ \substack{ p \leq y^* \\  p \not\in M } } d \vartheta_p   \bigg)\\
	& =\delta^{|M|}\sum_{\substack{p\leq y^*\\p\not\in M}}\sum_{k=1}^\infty\frac{1}{k^2p^{2k\sigma}} \\
	& \leq\delta^{|M|}\sum_{p>y } \frac{1}{p^{2\sigma}} \sum_{k=1}^\infty\frac{1}{k^2} \leq  c_2\delta^{|M|}y^{1-2\sigma}
\end{align*}
for  $ \sigma \geq \sigma'_1 $ and  some $ c_2 > 0$.
Hence, 
\begin{equation}\label{eqn:mean square c3}
\int_{C(\delta,M,T)}\iint_{D'}\left|\sum_{p\in M}\log\left(1-\frac{1}{p^{s+i\tau}}\right)-\sum_{p\leq y^*}\log\left(1-\frac{1}{p^{s+i\tau}}\right)\right|^2 ds d\tau\leq c_3 \delta^{|M|}y^{1-2\sigma'_1} H 
\end{equation}
for some $ c_3 >0$. 

Now, using the Montgomery-Vaughan inequality, we get
\begin{align*}
	&\int_T^{T+H}\left|\sum_{y^*<p\leq Y}\log\left(1-\frac{1}{p^{s+i\tau}}\right)\right|^2d\tau\\
	 &\qquad\qquad\qquad\qquad\qquad\leq 2\int_T^{T+H}\left|\sum_{y^*<p\leq Y}\frac{1}{p^{s+i\tau}}\right|^2d\tau+2\int_T^{T+H}\left|\sum_{y^*<p\leq Y}\sum_{k=2}^\infty\frac{1}{kp^{k(s+i\tau)}}\right|^2d\tau\\
	  &\qquad\qquad\qquad\qquad\qquad=2\sum_{y^*<p\leq Y}\frac{1}{p^{2\sigma}}(H+O(Y))  + O(H{y^*}^{1-2\sigma}) = O( H{y^*}^{1-2 \sigma} ) .
\end{align*}
Thus, by \eqref{choice of y star} and the above inequality  we have
\begin{equation}\label{eqn:mean square c4}
\begin{split}
\int_{C(\delta,M,T)} \iint_{D'} &  \left|\sum_{y^*<p\leq Y}\log\left(1-\frac{1}{p^{s+i\tau}}\right)\right|^2dsd\tau \\
& 
\leq \int_T^{T+H}\iint_{D'}\left|\sum_{y^*<p\leq Y}\log\left(1-\frac{1}{p^{s+i\tau}}\right)\right|^2dsd\tau\leq c_4  \delta^{|M|}y^{1-2\sigma'_1} H  
\end{split}\end{equation}
for some $ c_4 > 0 $. By \eqref{eqn:mean square c3} and \eqref{eqn:mean square c4} we have
\begin{equation}\label{eqn:C delta M T integral1}
\int_{C(\delta,M,T)} \iint_{D'}  \left|\sum_{\substack{ p \leq Y \\ p \not \in M} }\log\left(1-\frac{1}{p^{s+i\tau}}\right)\right|^2dsd\tau\leq c_5  \delta^{|M|}y^{1-2\sigma'_1} H 
\end{equation}
for some $ c_5 > 0$.

By \eqref{eqn:C delta M T integral1} we have
\[
\mu\left\{\tau\in C(\delta,M,T):\iint_{D'}\left|\sum_{\substack{p\leq Y\\p\not\in M}}\log\left(1-\frac{1}{p^{s+i\tau}}\right)\right|^2ds >  y^{\frac{1}{2}-\sigma'_1}\right\} \leq  c_5  \delta^{|M|}y^{\frac12 - \sigma'_1} H.
\]
For given $y$ we can choose  small $ \delta >0 $ such that $ c_5  \delta^{|M|}y^{\frac12 - \sigma'_1} < \frac12 $. Since $ \mu( C(\delta, M, T )) \sim \delta^{ |M|} H $ by \eqref{eqn:C delta asymp}, we have 
\[
\mu\left\{\tau\in C(\delta,M,T):\iint_{D'}\left|\sum_{\substack{p\leq Y\\p\not\in M}}\log\left(1-\frac{1}{p^{s+i\tau}}\right)\right|^2ds\leq y^{\frac{1}{2}-\sigma'_1}\right\} >\frac{1}{2}H\delta^{|M|}.
\]
It is well-known that
$$ |h(0)|^2 \leq \frac{1}{   \pi \ell^2 } \iint_{ |s| \leq \ell } |h(s)|^2 ds $$
for holomorphic function $h(s)$. Hence, we have
$$  \pi (R-r) \max_{s\in D}\left|\sum_{\substack{p\leq Y\\p\not\in M}}\log\left(1-\frac{1}{p^{s+i\tau}}\right)\right|^2 \leq \iint_{D'}\left|\sum_{\substack{p\leq Y\\p\not\in M}}\log\left(1-\frac{1}{p^{s+i\tau}}\right)\right|^2ds   $$
and  in consequence we have
\[
\mu\left\{\tau\in C(\delta,M,T): \max_{s\in D}\left|\sum_{\substack{p\leq Y\\p\not\in M}}\log\left(1-\frac{1}{p^{s+i\tau}}\right)\right|\leq \frac{1}{\sqrt{\pi}(R-r)}y^{\frac{1}{4}-\frac{1}{2}\sigma'_1}\right\}>\frac{1}{2}H\delta^{|M|}.
\]
Thus, for $y$ satisfying $ \frac{1}{\sqrt{\pi}(R-r)}y^{\frac{1}{4}-\frac{1}{2}\sigma'_1} \leq \varepsilon $, the set 
\begin{equation}\label{eqn:proof inequality 4}
E_T = \left\{\tau\in C(\delta,M,T):\max_{s\in D}\left|\sum_{p\in M}\log\left(1-\frac{1}{p^{s+i\tau}}\right)-\sum_{p\leq Y}\log\left(1-\frac{1}{p^{s+i\tau}}\right)\right|\leq \varepsilon\right\}
\end{equation}
has the measure $\mu(E_T)>\frac{1}{2}H\delta^{|M|}$.

By \eqref{eqn:proof inequality 1}, \eqref{eqn:proof inequality 2}, \eqref{eqn:proof inequality 3} and \eqref{eqn:proof inequality 4},    we have 
\[
\max_{s \in D} |\log\zeta(s+i\tau ) -  \log f(s)|< 4\varepsilon 
\]
for $\tau\in E_T $ and sufficiently large $T$. Thus, we have
	\[
	\liminf_{T\to\infty}\frac{1}{H}\mu\left\{\tau\in[T,T+H]: \max_{s\in D}| \log \zeta(s+i\tau) - \log f(s)|< 4 \varepsilon\right\} \geq \liminf_{T \to \infty} \frac{ \mu (E_T)}{H}   \geq \frac12 \delta^{|M|} >0,
	\]
which readily implies the theorem.

 \section{Proof of Theorem \ref{main theorem uncond}}\label{section proof 2}

Let  $ B> A> \frac{4}{ 2( \sigma_0 -r)-1}$. 
 By Lemma \ref{lem:shortDirichletpoly uncond}
   there exist $ c_1 >0$ and an increasing sequence $ T_j  \to \infty$ such that 
$$ \max_{s\in D} \left|\log\zeta(s+i\tau)-\sum_{p\leq Y_j }\log\left(1-\frac{1}{p^{s+i\tau}}\right)^{-1}\right| \leq (\log T_j )^{-c_1} $$
for  $ T_j \leq \tau \leq T_j + H_j $ as $T_j  \to \infty$, where
$Y_j = ( \log T_j )^A$  and $ H_j = ( \log T_j )^B$.  
We replace $H$, $T$ in the proof of Theorem \ref{main theorem} by $H_j$, $T_j$, respectively, except for \eqref{eqn:proof inequality 1},   then we find that
\[
	 \frac{1}{H_j }\mu\left\{\tau\in[T_j ,T_j +H_j ]: \max_{s\in D}| \log \zeta(s+i\tau) - \log f(s)|< 4 \varepsilon\right\} \geq  \frac{ \mu (E_{T_j} )}{H_j }   \geq \frac12 \delta^{|M|} .
	\]
This implies the theorem.

{
}

\end{document}